\documentclass[preprint,12pt]{elsarticle}

\usepackage{graphicx} % Required for inserting images
\usepackage{enumitem}
\usepackage{amssymb}
\usepackage{amsthm}
\usepackage{amsmath}
\usepackage{thmtools}
\usepackage[margin=1in]{geometry}
\usepackage[english]{babel}

\usepackage{tikz}
\usetikzlibrary{calc, arrows.meta}
\usepackage{caption, subcaption}
\usepackage{enumitem}
\usepackage{titlesec}

\titleformat{\subsection}{\bfseries\itshape}{\thesubsection.}{0.4em}{}

\theoremstyle{plain}

\newtheorem{theorem}{Theorem}[section]
\newtheorem{lemma}[theorem]{Lemma}
\newtheorem{corollary}[theorem]{Corollary}
\newtheorem{conjecture}[theorem]{Conjecture}
\newtheorem*{theorem*}{Theorem} %no numbering
\newtheorem{problem}[theorem]{Problem} %problem
\newtheorem*{problem*}{Problem} %problem

\DeclareMathOperator{\diam}{diam}

\newcommand{\sigmax}{\sigma^{-}_{\max}}
\DeclareMathOperator{\arrows}{Ar}
\DeclareMathOperator{\lines}{Li}

\journal{Discrete Mathematics}

\begin{document}

\begin{frontmatter}
\title{On the Total Regularity of Almost Mixed Moore Graphs}

\author{Ethan Shallcross}
\address{School of Mathematics and Statistics, The Open University, Milton Keynes, UK}
\ead{ershallcros@gmail.com}

\date{\today}

\begin{abstract}
The degree/diameter problem asks for the largest order of a graph with a given diameter and maximum vertex degree \cite{miller-2013}. This has been widely studied and given rise to a recent variation for mixed graphs (graphs with both undirected edges and directed arcs), where an additional bound is placed on the maximum directed out-degree of any vertex \cite{tuite-2019}. Both problems have applications to network design. Counting the possible number of vertices at each distance from a given vertex gives a bound on the order of a mixed graph satisfying the degree and diameter constraints (the \emph{mixed Moore bound}) \cite{buset-2016}. 
In this paper, we settle an open problem from \cite{tuite-2019} concerning the total regularity of mixed graphs whose order is one less than the mixed Moore bound (\emph{almost mixed Moore graphs}). We use this result to show that the three known almost mixed Moore graphs of diameter at least three are the only such mixed graphs.
\end{abstract}

\begin{keyword}
degree/diameter problem, mixed graphs, almost mixed Moore graphs
\end{keyword}

\end{frontmatter}

\section{Introduction}
Consider modelling a computer network using a (simple) graph, with vertices representing computers and edges representing communication channels. Transmissions may be sent between computers which are not directly connected through paths in the network. Due to hardware constraints, there may be a limit on how many connections a given computer may have. For fast transmission speeds, we may wish to limit the maximum distance between computers. Given these two constraints, how large (in terms of the number of computers) can the network be?
More generally, we may ask for the largest order of a graph with a given diameter and maximum vertex degree. This is known as the \emph{(undirected) degree/diameter problem}. The reader should see \cite{miller-2013} for a survey of the extensive literature.

In the preceding example, the communication channels were two-way (both linked computers could transmit over them). However, we may want some of the channels to be one-way (perhaps for security reasons). There may be a separate hardware implementation of one-way channels, giving rise to two different constraints on the number of outward connections a computer may have. As before, how large can the network be under our three constraints?

We may model this revised network using a mixed graph (a graph with both undirected edges and directed edges). In general, we now ask for the largest order of a mixed graph with given diameter and maximum undirected and directed degrees. This is known as the \emph{degree/diameter problem for mixed graphs} and is an active area of research (see \cite{nguyen-2007}, \cite{tuite-2019}, and \cite{tuite-2022}, for example).

Suppose we have a mixed graph $G$ satisfying our constraints. By counting the number of vertices permitted at each distance from a given vertex (up to the diameter), we obtain an upper bound on the order of $G$. This is known as the \emph{mixed Moore bound}. A mixed graph which meets this bound is known as a \emph{mixed Moore graph}.

Clearly, if we could find mixed Moore graphs for all possible constraint values, the problem would be solved. However, it was shown in \cite{nguyen-2007} that there are no such mixed graphs for diameter at least three. Therefore, the question becomes finding mixed graphs as close to the mixed Moore bound as possible.

Mixed graphs with order one less than the bound are called \emph{almost mixed Moore graphs}, and their existence is a natural subject of investigation. There have been recent papers on their properties, such as \cite{tuite-2019} and \cite{tuite-2022}. In particular, the \emph{total regularity} (defined in Section~\ref{sec:background}) of almost mixed Moore graphs was conjectured by Tuite and Erskine in 2019 \cite{tuite-2019}. The same authors made progress on this question, but up to now it remained open in general.

This paper proceeds as follows. The next section introduces the key terminology and notation which will be used throughout. We also formally introduce almost mixed Moore graphs and give several useful results from the literature. Section~\ref{sec:totalregularity} consists of the proof that all almost mixed Moore graphs are totally regular. In Section~\ref{sec:nonexistence}, we use our new result to show that there are no such graphs of diameter at least three other than those already discovered. We conclude with some open problems and potential directions for future study.

\section{Background}
\label{sec:background}
\subsection{Notation}
We begin with definitions and notation similar to that found in \cite{tuite-2022}. Formally, a mixed graph $G$ consists of a (finite) set $V(G)$ of vertices, an edge set $E(G)$ of unordered pairs $\{ x,y \}$ of distinct vertices (usually written $xy$), and an arc set $A(G)$ of ordered pairs $(x,y)$ of distinct vertices. We write $x \sim y$ and $x \rightarrow y$ to mean that $\{ x, y\} \in E(G)$ and $(x, y) \in A(G)$, respectively. We say that an edge $\{ u,v \}$ or arc $(u, v)$ is incident to $x$ if $x=u$ or $x=v$. By convention we allow \emph{at most one} edge or arc between two vertices. The \emph{order} of $G$ is $|V(G)|$ (the number of vertices in $G$). For any $S \subseteq V(G)$, the subgraph of $G$ \emph{induced} by $S$ is the mixed graph with vertex set $S$, edge set $E(G) \cap \{ \{ u, v \} : u, v \in S \}$ and arc set $A(G) \cap \{ (u, v) : u, v \in S \}$.

For any mixed graph $G$ and $v \in V(G)$, we define the \emph{directed in-} and \emph{directed out-neighbourhoods} of $v$ as the sets
\[ Z^-(v) = \{ u \in V(G) : u \rightarrow v \} \text{ and } Z^+(v) = \{ u \in V(G) : v \rightarrow u \}, \]
respectively. For $A \subseteq V(G)$, we write $Z^-(A)$ to mean $\bigcup_{v \in A} Z^-(v)$.
We also define the \emph{undirected neighbourhood} of $v$ as \[ U(v) = \{ u \in V(G) : u \sim v \}. \]Vertices in $Z^-(v)$ are known as directed in-neighbours of $v$ (undirected neighbours and directed out-neighbours are defined similarly). The \emph{in-} and \emph{out-neighbourhoods} of $v$ are then \[ N^-(v) = Z^-(v) \cup U(v) \text{ and } N^+(v) = Z^+(v) \cup U(v). \]

Next, the \emph{directed in-} and \emph{directed out-} and \emph{undirected degrees} of $v$ are $d^-(v) = |Z^-(v)|$, $d^+(v) = |Z^+(v)|$, and $d(v) = |U(v)|$ respectively. Continuing as in \cite{tuite-2022}, we shall say that $G$ is \emph{out-regular} if there exist integers $r$ and $z$ such that for all $v \in V(G)$, $d(v) = r$ and $d^+(v) = z$. In this case, $z$ and $r$ are referred to as the \emph{directed out-} and \emph{undirected degrees of $G$}, respectively. Similarly, $G$ is \emph{in-regular} if there exist integers $z$ and $r$ (the \emph{directed in-} and \emph{undirected degrees of $G$}, respectively) such that for all $v \in V(G)$, $d(v) = r$ and $d^-(v) = z$. $G$ is \emph{totally regular} if it is both out- and in-regular and for all $v \in V(G)$, $d^-(v) = d^+(v)=z$. The integer $z$ is known as the \emph{directed degree of $G$}.

If $G$ is out- or in-regular with undirected degree one, $v^*$ denotes the unique undirected neighbour of $v$. Sometimes $v^*$ refers to an arbitrary undirected neighbour of $v$ (when there is potentially more than one such neighbour), but this shall be clear from context.

A \emph{walk} of length $k$ is a sequence of vertices $v_0, v_1, v_2, \dots, v_k$ such that $v_i \sim v_{i+1}$ or $v_i \rightarrow v_{i+1}$ for $0 \leq i < k$. A \emph{path} is any walk in which all of the vertices are distinct. For any $u, v \in V(G)$, any walk $v=v_0, \dots, v_k=u$ is a $u,v$-walk and we similarly define a $u,v$-path. We define the \emph{distance} between $u$ and $v$, $d_G(u,v)$, to be the length of a shortest $u,v$-path in $G$ if such a path exists and $\infty$ otherwise. Where context is clear, we may simply write $d(u,v)$. We define the \emph{diameter} of $G$ as $\diam(G) = \max\{ d(u, v) : u, v \in V(G) \}$.

\subsection{The degree/diameter problem for mixed graphs}
The degree/diameter problem for mixed graphs asks for the largest order of a mixed graph with a given diameter $k$ such that all vertices in the graph have undirected degree at most $r$ and directed out-degree at most $z$. Suppose $G$ is a mixed graph with diameter $k$ such that each vertex has undirected degree at most $r$ and directed out-degree at most $z$. We first give a well-known upper bound on the order of $G$.

Let $u \in V(G)$. Similarly to as in \cite{tuite-2022} and \cite{buset-2016}, we may construct a \emph{mixed Moore tree} rooted at $u$ of depth $n$, denoted $^GT^n_u$. First, $^GT^0_u$ has no arcs or edges and the single vertex $(u)$ (a tuple consisting of $u$). 
To construct $^GT^{n+1}_u$ from $^GT^n_u$, for each tuple $t = (u_0, u_1, \dots, u_{n})$ of length $n+1$ in $V(^GT^n_u)$, for every $u_{n}^* \in U(u_n)$ except $u_{n-1}$ (if $n>0$), add the tuple $t' = (u_0, u_1, \dots, u_n, u_n^*)$ to the vertex set and $tt'$ to the edge set. Similarly, for each $u_n' \in Z^+(u_n)$, add the vertex $t' = (u_0, u_1, \dots, u_n, u_n')$ and an arc so that $t \rightarrow t'$. We usually identify the tuples with their final entries. If $G$ is clear from the context, we write $^GT^n_u$ as $T^n_u$ and $T_u$ to mean $^GT^{\diam(G)}_u$. 

The set of tuples of length $\ell$ in $T^n_u$ is \emph{layer $\ell - 1$} of $T^n_u$. For example, we may write that layer $0$ consists of $u$, layer $1$ as consists of the out-neighbours of $u$, and so on. For any $v \in N^+(u)$, the subgraph of $T_u^n$ induced by the tuples with first two entries $(u, v)$ is called the \emph{$v$-branch} of the mixed Moore tree. We say a vertex $w$ appears in layer $\ell$ of the $v$-branch to mean that (a tuple with final vertex) $w$ lies in both the $v$-branch and layer $\ell$ of $T^n_u$. A \emph{directed branch} is a branch of $T_u^n$ rooted at a directed out-neighbour of $u$. \emph{Undirected branches} are defined similarly.

Notice that a vertex $v \in V(G)$ may be the final entry in more than one tuple. In this case, we say that $v$ is \emph{repeated} and describe $v$ as appearing multiple times in the mixed Moore tree. A graphical representation will aid this explanation. 

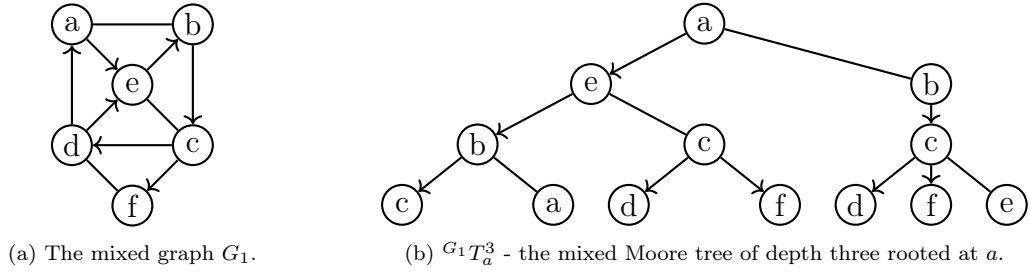
\begin{figure}[h!]
    \centering
    \begin{subfigure}{0.3\textwidth}
        \centering
        \begin{tikzpicture}[x=0.2mm,y=-0.2mm,inner sep=0.5mm,scale=1,thick,vertex/.style={circle,draw,minimum size=15}]
                \node at ($(0,0)$) [vertex] (a) {a};
                \node at ($(80, 0)$) [vertex] (b) {b};
                \node at ($(80, 80)$) [vertex] (c) {c};
                \node at ($(0, 80)$) [vertex] (d) {d};
                \node at ($(40, 40)$) [vertex] (e) {e};
                \node at ($(40, 120)$) [vertex] (f) {f};
            
                \path (a) edge (b);
                \path (a) edge[->] (e);
                \path (b) edge[->] (c);
                \path (c) edge (e);
                \path (c) edge[->] (d);
                \path (c) edge[->] (f);
                \path (d) edge (f);
                \path (d) edge[->] (a);
                \path (d) edge[->] (e);
                \path (e) edge[->] (b);
        \end{tikzpicture}
        \caption{The mixed graph $G_1$.}
    \end{subfigure}
    \begin{subfigure}{0.6\textwidth}
        \centering
        \begin{tikzpicture}[x=0.2mm,y=-0.2mm,inner sep=0.5mm,scale=1,thick,vertex/.style={circle,draw,minimum size=15}]
                \node at ($(0,0)$) [vertex] (a) {a};
                
                \node at ($(-75, 40)$) [vertex] (ae) {e};
                \node at ($(150, 40)$) [vertex] (ab) {b};

                \node at ($(-150, 80)$) [vertex] (aeb) {b};
                \node at ($(0, 80)$) [vertex] (aec) {c};
                \node at ($(150, 80)$) [vertex] (abc) {c};
                
                \node  at ($(-200, 120)$) [vertex] (aebc) {c};
                \node at ($(-100, 120)$) [vertex] (aeba) {a};
                \node at ($(-50, 120)$) [vertex] (aecd) {d};
                \node at ($(50, 120)$) [vertex] (aecf) {f};

                \node at ($(100, 120)$) [vertex] (abcd) {d};
                \node at ($(150, 120)$) [vertex] (abcf) {f};
                \node at ($(200, 120)$) [vertex] (abce) {e};

                \path (a) edge[->] (ae);
                \path (a) edge (ab);
                \path (ae) edge[->] (aeb);
                \path (ae) edge (aec);

                \path (ab) edge[->] (abc);
                \path (aeb) edge[->] (aebc);
                \path (aeb) edge (aeba);

                \path (aec) edge[->] (aecd);
                \path (aec) edge[->] (aecf);

                \path (abc) edge[->] (abcd);
                \path (abc) edge[->] (abcf);
                \path (abc) edge (abce);
        \end{tikzpicture}
        \caption{$^{G_1}T^3_a$ - the mixed Moore tree of depth three rooted at $a$.}
    \end{subfigure}
    \caption{A mixed graph and a mixed Moore tree.}
    \label{fig:mixedmooretreeexample}
\end{figure}

Figure~\ref{fig:mixedmooretreeexample} shows a mixed graph $G_1$ and $T_a^3$. Notice that $\diam(G_1) = 3$, so every vertex appears in the mixed Moore tree. There are unique undirected and directed branches (the $b$- and $e$-branches, respectively). We see that $f$ appears twice: in layer $3$ of both the undirected and directed branches. Notice that since the only occurrences of $f$ are in layer $3$, $d(a, f) = 3$.

In a similar vein, we introduce the notion of an \emph{inverse mixed Moore tree}. The inverse mixed Moore tree of depth $n$ rooted at $u$ is denoted $T_u^{-n}$. The construction is identical to that of mixed Moore trees, only we now consider in-neighbours of vertices as opposed to out-neighbours. We also use similar terminology as before. An example is shown in Figure~\ref{fig:inversemixedmooretreeexample}.

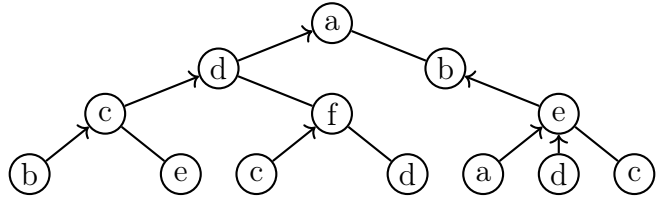
\begin{figure}[h!]
    \centering
    \begin{tikzpicture}[x=0.2mm,y=-0.2mm,inner sep=0.5mm,scale=1,thick,vertex/.style={circle,draw,minimum size=15}]
            \node at ($(0,0)$) [vertex] (a) {a};
            
            \node at ($(-75, 30)$) [vertex] (ad) {d};
            \node at ($(75, 30)$) [vertex] (ab) {b};

            \node at ($(-150, 60)$) [vertex] (adc) {c};
            \node at ($(0, 60)$) [vertex] (adf) {f};
            \node at ($(150, 60)$) [vertex] (abe) {e};
            
            \node  at ($(-200, 100)$) [vertex] (adcb) {b};
            \node at ($(-100, 100)$) [vertex] (adce) {e};
            \node at ($(-50, 100)$) [vertex] (adfc) {c};
            \node at ($(50, 100)$) [vertex] (adfd) {d};

            \node at ($(100, 100)$) [vertex] (abea) {a};
            \node at ($(150, 100)$) [vertex] (abed) {d};
            \node at ($(200, 100)$) [vertex] (abec) {c};

            \path (ad) edge[->] (a);
            \path (ab) edge (a);

            \path (adc) edge[->] (ad);
            \path (adf) edge (ad);
            \path (abe) edge[->] (ab);

            \path (adcb) edge[->] (adc);
            \path (adce) edge (adc);
            \path (adfc) edge[->] (adf);
            \path (adfd) edge (adf);
            \path (abea) edge[->] (abe);
            \path (abed) edge[->] (abe);
            \path (abec) edge (abe);
    \end{tikzpicture}
    \caption{The inverse mixed Moore tree $^{G_1}T^{-3}_a$.}
    \label{fig:inversemixedmooretreeexample}
\end{figure}

Notice that since $\diam(G) = k$, each vertex in $G$ appears in $T_u$. But given our degree constraints, we can find an upper bound on $|V(T_u)|$ (using a second-order recurrence relation, for example). This also gives an upper bound on the order of $G$, known as the \emph{mixed Moore bound}.
This extends the \emph{Moore bound}, which originally applied to undirected graphs \cite{miller-2013}. The new bound was introduced in \cite{buset-2016}, where the following theorem was proven.

\begin{theorem}[Mixed Moore bound {\cite[Theorem 1]{buset-2016}}]
    \label{thm:mixedmoorebound}
    Let $G$ be a mixed graph of diameter $k$ with maximum undirected degree $r$ and maximum directed out-degree $z$. Then, $G$ has order at most \[ M(r,z,k) = A \frac{u_1^{k+1} - 1}{u_1 - 1}  + B\frac{u_2^{k+1} - 1}{u_2 - 1}, \] where: 
    \begin{align*}
        v &= (z+r)^2 + 2(z-r) + 1, \quad u_1 = \frac{z+r-1 - \sqrt{v}}{2} \quad u_2 = \frac{z+r-1+\sqrt{v}}{2},\\ 
        A &= \frac{\sqrt{v} - (z+r+1)}{2\sqrt{v}}, \quad B = \frac{\sqrt{v} + (z+r+1)}{2\sqrt{v}}.
    \end{align*}
\end{theorem}

In particular, we have
\begin{equation}
    \label{eqn:mixedmoorek=2}
    M(r,z,2) = 1 + z + (r+z)^2.
\end{equation}

Mixed graphs which meet the mixed Moore bound with $r,z \geq 1$ are known as \emph{mixed Moore graphs}. Bos\'ak proved in \cite{bosak-1979} that there are infinitely many such graphs of diameter $k=2$ (in contrast with the undirected case \cite{hoffman-1960}) and gave necessary conditions for their existence. It is also proven there that all mixed Moore graphs of diameter two are totally regular.

Despite their abundance for $k=2$, Nguyen, Miller, and Gimbert proved in \cite{nguyen-2007} that there are no mixed Moore graphs for $k \geq 3$. In the same paper, the authors also completed the proof that all known mixed Moore graphs are unique (up to isomorphism) for each $(r,z)$ pair. So in the context of the degree/diameter problem, we next ask how close to the mixed Moore bound we may come.

If $G$ is a mixed graph with diameter $k$, maximum undirected degree $r$, and maximum directed out-degree degree $z$, the \emph{defect} of $G$ is defined as $\delta = M(r,z,k) - |V(G)|$. We say that $G$ is an $(r,z,k;-\delta)$-graph. This notation is as found in \cite{tuite-2019}, for example. The following result was proven in \cite{dalfo-2018}.

\begin{theorem}[See {\cite[Theorem 2.1]{dalfo-2018}}]
    \label{thm:undirectedbound}
    Any totally regular $(r,z,k;-\delta)$-graph with $k \geq 3$, has defect $\delta \geq r$.
\end{theorem}

Improving on this bound in the case of $r=z=1$ and $k>3$, the following result was then proven in \cite{tuite-2022}.

\begin{theorem}[See {\cite[Theorem 19]{tuite-2022}}]
    \label{thm:r=z=1defectbound}
    Any totally regular $(1,1,k; -\delta)$-graph with $k \geq 3$ has defect \[
        \delta \geq \sum_{t=1}^{k-2} Z_t' \left\lceil \frac{1}{3} + \frac{1}{3} \left\lfloor \frac{k-t}{2} \right\rfloor \right\rceil,
    \] where $Z'_1 = 1$ and for $2 \leq t \leq k-2$ \[
        Z'_t = \frac{1}{2^{t-1}\sqrt{5}}\left(\left(1+\sqrt{5}\right)^{t-1} - \left(1-\sqrt{5}\right)^{t-1}\right).
    \]
\end{theorem}

\begin{figure}[h!]
    \centering
    \begin{subfigure}{0.24\textwidth}
        \centering
        \begin{tikzpicture}[x=0.2mm,y=-0.2mm,inner sep=0.5mm,scale=1,thick,vertex/.style={circle,draw,minimum size=10}]
            \foreach \c [count=\i, evaluate=\i as \j using {int(\i-1)}] in {0,...,4}{
                \node at ($(0,0)!8!\j*360/5+90:(9,0)$) [vertex] (\c) {};
            }
            \foreach \c [count=\i, evaluate=\i as \j using {int(\i-1)}] in {5,...,9}{
                \node at ($(0,0)!4!\j*360/5+90:(9,0)$) [vertex] (\c) {};
            }
            \path (0) edge[->] (9);
            \path (1) edge[->] (5);
            \path (2) edge[->] (6);
            \path (3) edge[->] (7);
            \path (4) edge[->] (8);

            \path (5) edge[->] (4);
            \path (6) edge[->] (0);
            \path (7) edge[->] (1);
            \path (8) edge[->] (2);
            \path (9) edge[->] (3);
            
            \path (0) edge (1);
            \path (1) edge (2);
            \path (2) edge (3);
            \path (3) edge (4);
            \path (4) edge (0);
            
            \path (5) edge (8);
            \path (5) edge (7);
            \path (6) edge (9);
            \path (6) edge (8);

            \path (7) edge (9);

        \end{tikzpicture}
        \caption{See \cite{buset-2017}}
        \label{fig:diamtwommg}
    \end{subfigure}
    \begin{subfigure}{0.24\textwidth}
        \centering
        \begin{tikzpicture}[x=0.2mm,y=-0.2mm,inner sep=0.5mm,scale=1,thick,vertex/.style={circle,draw,minimum size=10}]
	        \foreach \c [count=\i, evaluate=\i as \j using {int(\i-1)}] in {0,...,4}{
		        \node at ($(0,0)!8!\j*360/5+90:(9,0)$) [vertex] (\c) {};
	        }
            \foreach \c [count=\i, evaluate=\i as \j using {int(\i-1)}] in {5,...,9}{
		        \node at ($(0,0)!4!\j*360/5+90:(9,0)$) [vertex] (\c) {};
	        }
            \path (0) edge (5);
            \path (1) edge (6);
            \path (2) edge (7);
            \path (3) edge (8);
            \path (4) edge (9);
            
            \path (0) edge[->] (1);
            \path (1) edge[->] (2);
            \path (2) edge[->] (3);
            \path (3) edge[->] (4);
            \path (4) edge[->] (0);
            
            \path (5) edge[->] (9);
            \path (9) edge[->] (8);
            \path (8) edge[->] (7);
            \path (7) edge[->] (6);
            \path (6) edge[->] (5);

        \end{tikzpicture}
        \caption{See \cite{dalfo-2018}}
    \end{subfigure}
    \begin{subfigure}{0.22\textwidth}
        \centering
        \begin{tikzpicture}[x=0.2mm,y=-0.2mm,inner sep=0.5mm,scale=1,thick,vertex/.style={circle,draw,minimum size=10}]
	        \foreach \c [count=\i, evaluate=\i as \j using {int(\i-1)}] in {0,...,4}{
		        \node at ($(0,0)!8!\j*360/5+90:(9,0)$) [vertex] (\c) {};
	        }
            \foreach \c [count=\i, evaluate=\i as \j using {int(\i-1)}] in {5,...,9}{
		        \node at ($(0,0)!4!\j*360/5+90:(9,0)$) [vertex] (\c) {};
	        }
            \path (0) edge (5);
            \path (1) edge (6);
            \path (2) edge[->] (7);
            \path (4) edge (9);
            \path (8) edge[->] (3);
            
            \path (0) edge[->] (1);
            \path (1) edge[->] (2);
            \path (2) edge (3);
            \path (3) edge[->] (4);
            \path (4) edge[->] (0);
            
            \path (5) edge[->] (9);
            \path (9) edge[->] (8);
            \path (8) edge (7);
            \path (7) edge[->] (6);
            \path (6) edge[->] (5);
        \end{tikzpicture}
        \caption{See \cite{dalfo-2018}}
    \end{subfigure}
    \begin{subfigure}{0.22\textwidth}
        \centering
        \begin{tikzpicture}[x=0.2mm,y=-0.2mm,inner sep=0.5mm,scale=1,thick,vertex/.style={circle,draw,minimum size=10}]
	        \foreach \c [count=\i, evaluate=\i as \j using {int(\i-1)}] in {0,...,4}{
		        \node at ($(0,0)!8!\j*360/5+90:(9,0)$) [vertex] (\c) {};
	        }
            \foreach \c [count=\i, evaluate=\i as \j using {int(\i-1)}] in {5,...,9}{
		        \node at ($(0,0)!4!\j*360/5+90:(9,0)$) [vertex] (\c) {};
	        }
            \path (5) edge[->] (0);
            \path (1) edge (6);
            \path (2) edge[->] (7);
            \path (8) edge[->] (3);
            \path (4) edge[->] (9);
            
            \path (0) edge[->] (1);
            \path (1) edge[->] (2);
            \path (2) edge (3);
            \path (3) edge[->] (4);
            \path (4) edge (0);
            
            \path (5) edge (9);
            \path (9) edge[->] (8);
            \path (8) edge (7);
            \path (7) edge[->] (6);
            \path (6) edge[->] (5);

        \end{tikzpicture}
        \caption{See \cite{dalfo-2018}}
    \end{subfigure}
    \caption{The only known almost mixed Moore graphs.}
    \label{fig:almostmixedmooregraphs}
\end{figure}
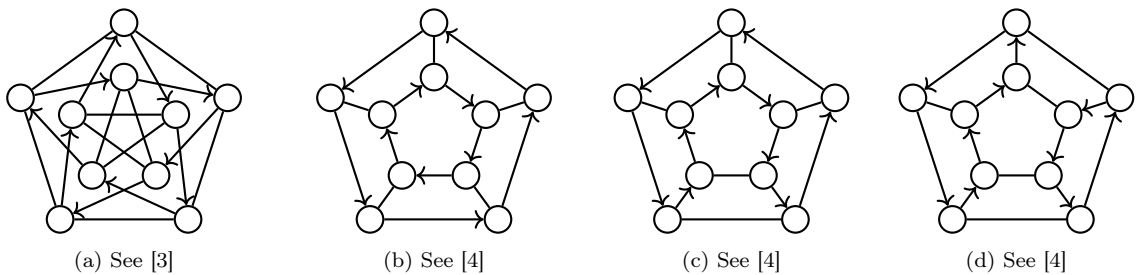

Since we know that there are no mixed Moore graphs of diameter at least three, it is natural to investigate the existence of mixed graphs with defect one. When $r,z \geq 1$, these are known as \emph{almost mixed Moore graphs}. Up to now, only four such graphs have been discovered. These are shown in Figure~\ref{fig:almostmixedmooregraphs}. The mixed graph in Figure~\ref{fig:diamtwommg} is a $(2,1,2;-1)$-graph. This was proven to be the unique such mixed graph in \cite{buset-2017}. The other three mixed graphs are $(1,1,3;-1)$-graphs. They were proven to be the only such mixed graphs in \cite{dalfo-2018}.

\newpage
\begin{lemma}[{See \cite[Lemma 2.1]{tuite-2019}}]
    Any $(r,z,k;-\delta)$-graph with $r, z \geq 1$, $k \geq 2$, and $\delta<r+z$ is out-regular with undirected degree $r$ and directed out-degree $z$.
    \label{lem:outreg}
\end{lemma}
\begin{proof}
    Let $G$ be any $(r,z,k;-\delta)$-graph with $\delta<r+z$ and suppose for a contradiction that $G$ is not out-regular. Then, we may choose a vertex $v \in V(G)$ with $|N^+(v)| < r+z$. 
    Hence, $T^1_u$ and $T^2_u$ include at most $M(r,z,1) - 1$ and $M(r,z,2) - 1 - (r-1) - z$ vertices, respectively. Thus, $|V(T_u)| \leq M(r,z,k) - (r+z)$.
    But all vertices of $G$ appear in $T_u$, contradicting that $\delta < r+z$.
\end{proof}

Let $G$ be any $(r,z,k;-1)$-graph with $r, z \geq 1$ and $k \geq 2$. For any $u \in V(G)$, each vertex of $G$ appears in $T_u$. By Lemma~\ref{lem:outreg}, $G$ is out-regular, so $|V(T_u)| = M(r, z, k)$. Therefore, as $G$ has order $M(r, z, k) - 1$, there must be a unique vertex in $G$ which appears twice in $T_u$. We shall call this vertex the \emph{repeat} of $u$ and denote it by $r(u)$. If $r(u) = u$, we say that $u$ is \emph{self-repeat}. The function mapping vertices in $G$ to their repeats is the \emph{repeat function}. 

\begin{lemma}
    \label{lem:repperm}
    The repeat function of any totally regular almost mixed Moore graph is a bijection.
\end{lemma}
\begin{proof}
    Let $G$ be a totally regular $(r,z,k;-1)$-graph and $u \in V(G)$. As in the case of (non-inverse) mixed Moore trees, $|V(T^{-k}_u)| = M(r, z, k)$. Since $\diam(G) = k$, each vertex must appear in $T^{-k}_u$. But $G$ has order $M(r, z, k) - 1$. Therefore, exactly one vertex of $G$ must be appear exactly twice in the tree, with all other vertices appearing exactly once. A vertex $v$ is repeated in $T^{-k}_u$ if and only if $r(v) = u$. Thus, exactly one vertex in $G$ has repeat $u$.
\end{proof}

Looking at Lemma~\ref{lem:outreg} all known examples led Tuite and Erskine to conjecture in 2019 that all $(r, z, k; -1)$- and $(r, z, k; +1)$-graphs are totally regular \cite[Conjecture 7.1]{tuite-2019}. 
In \cite{tuite-2022}, the same authors proved the result for the latter class of mixed graphs. We have not defined $(r, z, k; +1)$-graphs here, but the interested reader should consult \cite{tuite-2019} or \cite{tuite-2022} for more information. We are left with the following unproven conjecture.

\begin{conjecture}[See {\cite[Conjecture 7.1]{tuite-2019}}]
    \label{conj:totallyregular}
    All almost mixed Moore graphs are totally regular.
\end{conjecture}

Progress was made in \cite{tuite-2019}, where counting arguments and spectral theory were used to prove the following two results.

\begin{theorem}[See {\cite[Theorem 3.5]{tuite-2019}}]
    \label{thm:(1,1,3;-1)totallyregular}
    $(1,1,k;-1)$-graphs are totally regular for $k \geq 3$.
\end{theorem}

\begin{theorem}[See {\cite[Theorem 2.8]{tuite-2019}}]
    \label{thm:diam2totallyregular}
    All almost mixed Moore graphs with diameter two are totally regular.
\end{theorem}

Finally, we give some more definitions from \cite{tuite-2019} which will be useful in the next section. First, if $G$ is an $(r,z,k;-\delta)$-graph, we define \[ 
    S_G = \{ v \in V(G) : d^-(v) < z \} \text{ and } S'_G = \{ v \in V(G) : d^-(v) > z \},
\]
so that $G$ is totally regular if and only if $S_G = S'_G = \emptyset$. For $v \in V(G)$, we set $\sigma^-(v) = z - d^-(v)$ and $\sigma^+(v) = d^-(v) - z$. Also, ${\sigmax(G) = \max\{ \sigma^-(v) : v \in V(G) \}}$.

For $u \in V(G)$, $R(u)$ is the multiset consisting of $n_v - 1$ occurrences of each $v \in V(G)$, where $n_v$ is the number of times $v$ appears in $T_u$. For any set $S$, $|S \cap R(u)|$ denotes the number of elements in $R(u)$ also in $S$, counting repeats. For example, if $R(u) = \{ 0,0,1,2,3 \}$ and $S = \{ 0,1 \}$, then $|S \cap R(u)| = 3$.

\section{The total regularity of almost mixed Moore graphs}
\label{sec:totalregularity}
\subsection{Technical lemmas and definitions}
First, we state as a lemma a result first given in the proof of \cite[Theorem 15]{tuite-2022}.
\begin{lemma}
    \label{lem:sigmasums}
    For any $(r,z,k;-\delta)$-graph $G$ with $r, z \geq 1$, $k \geq 2$, $\delta < r+z$,
    \[ \sum_{v \in S_G} \sigma^-(v) = \sum_{v' \in S'_G} \sigma^+(v').\]
\end{lemma}\begin{proof}
By Lemma~\ref{lem:outreg}, $G$ is out-regular. Therefore, 
\[ \sum_{v \in V(G)} d^-(v) = \sum_{v\in V(G)} d^+(v) = z|V(G)| .\]
It follows that 
\begin{align*}
    \sum_{v' \in S'_G} \sigma^+(v') - \sum_{v \in S_G} \sigma^-(v) &= \sum_{v' \in S'_G} (d^-(v') - z) - \sum_{v \in S_G} (z - d^-(v)) \\
    &= \sum_{v \in V(G)} d^-(v) - z|V(G)| = 0.
\end{align*}
\end{proof}

We know that if an $(r,z,k;-\delta)$-graph $G$ is not totally regular, then at least one of $S_G$ and $S'_G$ is non-empty. The above result now allows us to say more.

\begin{corollary}
    \label{cor:snonempty}
    Let $r, z \geq 1$, $k \geq 2$, $\delta < r+z$. If $G$ is an $(r,z,k;-\delta)$-graph which is not totally regular, then $S_G$ is non-empty.
\end{corollary}

As in \cite{tuite-2022}, an \emph{arrow vertex} (with respect to $u$) is any vertex in layer $2 \leq t < k$ of an \emph{undirected} branch of $T_u$ with a \emph{directed} in-neighbour in layer $t-1$ of the same branch. A \emph{line vertex} (with respect to $u$) is any vertex in layer $2 \leq t < k$ of a \emph{directed} branch of $T_u$ with an \emph{undirected} neighbour in layer $t-1$ of the same branch. We let $\arrows(u)$ and $\lines(u)$ denote the set of arrow and line vertices with respect to $u$, respectively.

\begin{lemma}
    \label{lem:undirectedindeg}
    Let $G$ be an $(r, z, k; -\delta)$-graph with $k \geq 2$, $r, z \geq 1$, and $\delta < r+z$. If $u, v \in V(G)$ and $v$ appears in an undirected branch of $T_u$, then $d^-(v) \geq z-|N^-(v) \cap R(u)|$. Moreover, if $v \in \arrows(u)$, then $d^-(v) \geq z+1-|N^-(v) \cap R(u)|$.
\end{lemma}
\begin{proof}
    Let $G$ be an $(r, z, k; -\delta)$-graph with $k \geq 2$, $r, z \geq 1$, and $\delta < r+z$. By Lemma~\ref{lem:outreg}, $G$ is out-regular. Let $u, v \in V(G)$ be such that $v$ appears in an undirected branch of $T_u$. Let $\ell$ be the smallest integer such that $v$ appears in layer $\ell$ of an undirected branch. 
    We first show that there are at least $r+z$ appearances of in-neighbours of $v$ in $T_u$ (including repetitions).
    \begin{itemize}
    \item If $\ell=1$, then $u \in N^-(v)$. Also, $r-1$ in-neighbours of $v$ appear in the $v$-branch of $T_u$. Since $d(w, v) \leq k$ for each $w \in Z^+(u)$, there must be an in-neighbour of $v$ appearing in each directed branch. Hence, there are at least $r+z$ appearances of in-neighbours of $v$ in $T_u$. 

    \item If $1 < \ell < k$, then there are at least $r$ in-neighbours of $v$ in one of the undirected branches of $T_u$. Additionally, at least one in-neighbour of $v$ appears in each of the $z$ directed branches. 
    
    If $v \in \arrows(u)$, then there are at least $r+1$ in-neighbours of $v$ in one of the undirected branches. Hence, there are at least $r+z+1$ appearances of in-neighbours of $v$ in $T_u$. 
    
    \item Finally, suppose that $\ell = k$. If $d(u,v) = k$, then at least one in-neighbour of $v$ appears in each of the $r+z$ branches of $T_u$. If $1 \leq d(u, v) < k$, then $v$ and at least $r$ of its in-neighbours appear in one of the directed branches. Also, at least one in-neighbour of $v$ appears in each of the $z-1$ other directed branches and in one undirected branch. If $u=v$, then $r$ in-neighbours of $v$ appear in layer $1$ and at least one in-neighbour appears in each of the $z$ directed branches.
    \end{itemize}
    Accounting for repetitions, it follows that $d(v) + d^-(v) \geq r+z - |N^-(v) \cap R(u)|$. But $G$ is out-regular, so $d(v) = r$ and $d^-(v) \geq z - |N^-(v) \cap R(u)|$. If $v \in \arrows(u)$, then we may further conclude that $d^-(v) \geq z + 1 - |N^-(v) \cap R(u)|$.
\end{proof}

\begin{lemma}
    \label{lem:arrowsinclusion}
    Let $G$ be an $(r,z,k;-1)$-graph with $r,z \geq 1$, $k \geq 3$. For all $u \in V(G)$, $\arrows(u) \subseteq (N^+(r(u)) \setminus S_G) \cup S'_G$.
\end{lemma}
\begin{proof}
    Let $G$ be as described, $u \in V(G)$, and $v \in \arrows(u)$. Since $r + z > 1$, Lemma~\ref{lem:undirectedindeg} gives $d^-(v) \geq z+1-|N^-(v) \cap R(u)|$. Notice that $R(u) = \{ r(u) \}$. If $r(u) \notin N^-(v)$, then $d^-(v) \geq z+1$ and $v \in S'_G$. Otherwise, $d^-(v) \geq z$ and $v \in N^+(r(u)) \setminus S_G$.
\end{proof}

\subsection{$(r, z, k; -1)$-graphs with $r \geq 1$, $z \geq 2$, $k \geq 3$}

We proceed to prove that all almost mixed Moore graphs with directed degree at least two and diameter at least three are totally regular. This result will be deduced from a new bound on the defect of a mixed graph. Let $G$ be an $(r,z,k;-\delta)$-graph with $r, z \geq 1$, and $k \geq 3$.

\begin{lemma}
    If $\delta < z$, then for each $v \in V(G)$, $d^-(v) > 0$.
    \label{lem:vindeg}
\end{lemma}
\begin{proof}
    Since $\delta < z$, $G$ is out-regular by Lemma~\ref{lem:outreg}.
    Since $r \geq 1$, we may choose $v^* \in U(v)$.
    Then, $v$ lies in an undirected branch of $T_{v^*}$ and by Lemma~\ref{lem:undirectedindeg}, $d^-(v) \geq z - |N^-(v) \cap R(v^*)|$. But $|N^-(v) \cap R(v^*)| \leq \delta < z$, so $d^-(v) > 0$.
\end{proof}

\begin{theorem}
    \label{thm:defectboundz>1}
    Either $\delta = z = \sigmax(G)$ or $\delta > \sigmax(G)$.
\end{theorem}
\begin{proof}
    If $\delta \geq z$, then $\delta > \sigmax(G)$ unless $\delta = z = \sigmax(G)$. Now suppose that $\delta < z$ and choose $v \in V(G)$ such that $\sigma^-(v) = \sigmax(G)$. By Lemmas~\ref{lem:outreg} and \ref{lem:vindeg}, we may choose $u \in V(G)$ such that $u \sim w \rightarrow v$ for some $w \in V(G) \setminus \{ u, v \}$. Then, $v \in \arrows(u)$ and by Lemma~\ref{lem:undirectedindeg}, $d^-(v) \geq z+1-|N^-(v) \cap R(u)|$. But $|N^-(v) \cap R(u)| \leq \delta$, so $\delta > z - d^-(v) = \sigmax(G)$.
\end{proof}

Notice that the above theorem also shows that there are no mixed Moore graphs for $r,z \geq 1$ and $k \geq 3$.
If $G$ is not totally regular, Corollary~\ref{cor:snonempty} gives $\sigmax(G) \geq 1$. It then follows from the above result that $\delta > 1$ whenever $z>1$.

\begin{corollary}
    \label{thm:totallyregz>1}
    All $(r,z,k;-1)$-graphs with $r \geq 1$, $z \geq 2$, $k \geq 3$ are totally regular.
\end{corollary}

\subsection{$(r,z,k;-1)$-graphs with $r \geq 2, z=1, k \geq 3$}
\label{subsec:z=1totalregularity}
We now complete the proof of Conjecture~\ref{conj:totallyregular}. In this section, let $G$ be any $(r,1,k;-1)$-graph with $r \geq 2$, $k \geq 3$.

\begin{lemma}
    \label{lem:nonlinevertex}
    For all $u \in V(G)$, $S_G \setminus \lines(u) \subseteq N^+(r(u))$. Furthermore, if there exists $v \in S_G \setminus \lines(u)$ with $d(u,v) < k$, then $r(u)$ lies in layer $k$ of the directed branch of $T_u$.
\end{lemma}
\begin{proof}
Let $u \in V(G)$ and $v \in S_G \setminus \lines(u)$, with $v \neq u$. By Lemma~\ref{lem:outreg}, $d(v) = r$. As $d^-(v) = 0$ and $v \notin \lines(u)$, $v$ appears in an undirected branch of $T_u$ or in layer $k$ of the directed branch. Since $d(u^*, v) \leq k$ for each $u^* \in U(u)$, if $v$ does not appear in an undirected branch, then it has an in-neighbour in each of these branches. Hence, there are at least $r+1$ appearances of in-neighbours of $v$ in $T_u$ and (as reasoned in the proof of Lemma~\ref{lem:undirectedindeg}) $d^-(v) \geq 1 - |N^-(v) \cap R(u)|$. If $v$ appears in an undirected branch, then Lemma~\ref{lem:undirectedindeg} gives the same result.
But $d^-(v) = 0$, so $r(u) \in N^-(v)$. 

Suppose additionally that we may choose $v$ such that $1 \leq d(u, v) < k$. Now, $v$ lies in an undirected branch of $T_u$ outside of layer $k$. Thus, there must be $r$ in-neighbours of $v$ which are either equal to $u$ or in an undirected branch of $T_u$. To prevent more than one repetition, there must be precisely one in-neighbour of $v$ in the directed branch. This vertex is $r(u)$. If $r(u)$ did not appear in layer $k$ of the directed branch, then $v$ would also be repeated in $T_u$ --- a contradiction.

Finally, suppose that $u \in S_G \setminus \lines(u)$. 
Each in-neighbour of $u$ appears in layer $1$ of $T_u$. Since $\diam(G) = k$ and $G$ has defect one, $r(u)$ is an in-neighbour of $u$ and appears in layer $k$ of the directed branch.
\end{proof}

\begin{theorem}
    \label{thm:totallyregz=1}
    $G$ is totally regular.
\end{theorem}
\begin{proof}
By Lemma~\ref{lem:outreg}, $G$ is out-regular. Suppose for a contradiction that $G$ is not totally regular. Then, by Corollary~\ref{cor:snonempty}, we may choose $v \in S_G$. Since $r \geq 2$, we may choose $(v^*)^* \in V(G) \setminus \{ v \}$ such that $(v^*)^* \sim v^* \sim v$ for some $v^* \in U(v)$.

For $x \in V(G)$, let $\arrows(r^{-1}(x))$ be the set of vertices which are arrow vertices with respect to a vertex whose repeat is $x$. By Lemma~\ref{lem:arrowsinclusion}, $\arrows(r^{-1}(x)) \subseteq (N^+(x) \setminus S_G) \cup S'_G$. Therefore, 
\begin{align} \sum_{v' \in S'_G} \sigma^+(v') \geq |S'_G| &\geq |\arrows(r^{-1}(x)) \setminus (N^+(x) \setminus S_G)| \nonumber \\ 
&\geq |\arrows(r^{-1}(x))| - |N^+(x) \setminus S_G|. 
\label{eqn:excessbound}
\end{align}

Let $F = \{ (v^*)^* \} \cup U((v^*)^*)$. If $v \in \lines(w)$ for some $w \in F$, then both $v$ and $v^*$ are repeated in $T_w$ --- a contradiction. So by Lemma~\ref{lem:nonlinevertex}, $r(w) \in N^-(v)$ for all $w \in F$. Since $d^-(v) = 0$, $|N^-(v)| = r$. Therefore, as $|F| = r+1 > r$, there must be distinct $w_1, w_2 \in F$ and $u \in N^-(v)$ such that $r(w_1) = r(w_2) = u$.

Again by Lemma~\ref{lem:nonlinevertex}, as $d((v^*)^*, v) = 2 < k$, all vertices outside of layer $k$ in the directed branch of $T_{(v^*)^*}$ are distinct. Therefore, the vertices in layers $2, \dots, k-1$ in the directed branches of $T_{w_1}$ and $T_{w_2}$ are distinct.
Hence, $\lines(w_1) \cap \lines(w_2) = \emptyset$ and
by Lemma~\ref{lem:nonlinevertex}, $S_G \subseteq N^+(u)$. Then, by (\ref{eqn:excessbound}) and the out-regularity of $G$,
\begin{equation}
\sum_{v' \in S'_G} \sigma^+(v') \geq |\arrows(r^{-1}(u))| - (|N^+(u)| - |S_G|) = |\arrows(r^{-1}(u))| - r-1 + |S_G|.
\label{eqn:sig+bound}
\end{equation}
We now show that $|\arrows(r^{-1}(u))| \geq r+2$. Since $G$ has defect one, $|\arrows(w_1)| \geq r$. First, suppose $w_1, w_2 \neq (v^*)^*$. All vertices not in layer $k$ of the directed branch of $T_{(v^*)^*}$ are distinct. Thus, if $a \in \arrows(w_2)$ is such that $d(w_2, a) \geq k-2$ and $a$ doesn't appear in the $(v^*)^*$-branch of $T_{w_2}$, $a \notin \arrows(w_1)$. If $k>3$ or $r>2$, there are at least two such arrow vertices and $|\arrows(r^{-1}(u))| \geq r+2$.

If $w_1 = (v^*)^*$, the directed out-neighbour of $w_1$ is an arrow vertex with respect to $w_2$ but not $w_1$. Also, there is an arrow vertex with respect to $w_2$ in layer $k$ of $T_{(v^*)^*}$ which cannot be an arrow vertex with respect to $w_1$.
So if $r \geq 3$, $k \geq 4$, or $w_1 = (v^*)^*$, we conclude that $|\arrows(r^{-1}(u))| \geq r+2$. Then, Equation~\ref{eqn:sig+bound} gives \[\sum_{v' \in S'_G} \sigma^+(v') > |S_G| = \sum_{v \in S_G} \sigma^-(v), \] which contradicts Lemma~\ref{lem:sigmasums}.

This leaves the case of $r=2$, $k=3$, and $w_1, w_2 \neq (v^*)^*$. Now, $\left|S_G\right| \leq \left|N^+(u)\right| = 3$. As well as this, for all $x \in V(G)$, $v \in \lines(x)$ if and only if $x \in Z^-(U(v))$. We have $d(v) = 2$. If both (undirected) neighbours of $v$ lie in $S'_G$, Lemma~\ref{lem:sigmasums} gives
\begin{align} 
    \label{eqn:S'bound}
    \left|Z^-(U(v))\right| \leq \sum_{v' \in S'_G} d^-(v') &= \sum_{v' \in S'_G} \sigma^+(v') + \left|S'_G\right| \nonumber \\
&\leq 2 \sum_{v' \in S'_G} \sigma^+(v') = 2\left|S_G\right| \leq 6.
\end{align}
If only one neighbour of $v$ lies in $S'_G$, then $\left|Z^-(U(v))\right| \leq \sum_{v' \in S'_G} d^-(v') + 1 \leq 7$. If no neighbour of $v$ lies in $S'_G$, then $\left|Z^-(U(v))\right| \leq 2$. So $v$ is a line vertex with respect to at most seven vertices in $G$.
As $|V(G)| = M(2,1,3) - 1 = 27$ (from Equation~\ref{eqn:mixedmoorek=2}), by Lemma~\ref{lem:nonlinevertex} there are at least $20$ vertices in $G$ whose repeat lies in $N^-(v)$. But $|N^-(v)| = 2$, so may we choose $y \in N^-(v)$ such that at least $10$ vertices of $G$ have repeat $y$. 

Since $G$ has defect one, there are two arrow vertices with respect to any vertex. Letting $B = \{ (a, w) : r(w) = y \text{ and } a \in \arrows(w) \}$, we have $|B| \geq 2\times 10 = 20$. Again since $G$ has defect one, any $a \in V(G)$ is an arrow vertex with respect to exactly $2d^-(a)$ vertices. Thus, by (\ref{eqn:S'bound}), in at least $20 - \sum_{v' \in S'_G} 2d^-(v') \geq 8$ of the pairs $(a,w) \in B$, $a \notin S'_G$. But by Lemma~\ref{lem:arrowsinclusion}, all such $a$ lie in $N^+(y)$. Thus, there can be at most $2 \times 3 = 6$ such pairs --- a contradiction.
\end{proof}

As promised, we now answer Conjecture~\ref{conj:totallyregular} in the affirmative.

\begin{theorem}
All almost mixed Moore graphs are totally regular.
\label{thm:totallyregsummary}
\end{theorem}
\begin{proof}
Corollary~\ref{thm:totallyregz>1} and Theorem~\ref{thm:totallyregz=1} show that all $(r,z,k;-1)$-graphs with $k \geq 3$, $r,z \geq 1$, and $r + z > 2$ are totally regular. The result for $r=z=1$, $k \geq 3$ is Theorem~\ref{thm:(1,1,3;-1)totallyregular} and the result for $k=2$ is Theorem~\ref{thm:diam2totallyregular}. 

By definition, if $G$ is an almost mixed Moore graph, there exists $v \in V(G)$ with $d^+(v) \geq 1$. Since we only permit a single arc or edge between any pair of vertices, $d(w, v) > 1$ for all $w \in Z^+(v)$. Hence, there are no almost mixed Moore graphs of diameter one.
\end{proof}

\section{The existence of almost mixed Moore Graphs}
\label{sec:nonexistence}
We now investigate the existence of almost mixed Moore graphs. Recall from Theorem~\ref{thm:undirectedbound} that there are no totally regular $(r, z, k; -1)$-graphs whenever $k \geq 3$, $r > 1$, and $z \geq 1$.
Theorem~\ref{thm:totallyregsummary} therefore has following corollary, ruling out the existence of a large class of possible almost mixed Moore graphs.

\begin{corollary}
    \label{cor:diam>=3r=1}
    Any almost mixed Moore graph with diameter at least three is totally regular with undirected degree one.
\end{corollary}

Therefore, if there exists an $(r,z,k;-1)$-graph (with $r,z \geq 1$), we must either have $k \geq 3$ and $r=1$ or $k=2$. We first prove that there are no $(r,z,k;-1)$-graphs with $k \geq 4$ and $r=1$.

\begin{theorem}
    \label{thm:almostmixedmoorenonexistencek>=4}
    There is no $(r, z, k; -1)$-graph for $r, z \geq 1$, $k \geq 4$.
\end{theorem}
\begin{proof}
Suppose for a contradiction that there exists an $(r, z, k; -1)$-graph $G$ for $r, z \geq 1$ and $k \geq 4$. By Corollary~\ref{cor:diam>=3r=1}, $G$ is totally regular and $r=1$. By Theorem~\ref{thm:r=z=1defectbound}, $z \geq 2$. Since $k \geq 4$, there are arrow vertices in layers $2$ and $3$ of $T_u$. By Lemma~\ref{lem:arrowsinclusion}, all have in-neighbour $r(u)$. We see that there are at least two paths of length at most $k$ from $r(u)$ to each of the arrow vertices in layer $3$ of $T_u$ which are a directed out-neighbour of an arrow vertex in layer $2$. However, since $z>1$ and $G$ has defect one, there are at least two such arrow vertices. Therefore, there are at least two vertices repeated in $T_{r(u)}$ --- a contradiction.
\end{proof}

We now turn to the question of the existence of $(1,z,3;-1)$-graphs. We begin by recalling a result from \cite[Theorem 2.2(c)]{dalfo-2024}. This originally only applied to totally regular almost mixed Moore graphs of diameter three, but the regularity condition is now unnecessary.

\begin{lemma}[See {\cite[Theorem 2.2(c)]{dalfo-2024}}]
    Let $G$ be a $(1, z, 3; -1)$-graph with $z \geq 1$. If $r(u) = u^*$ for all $u \in V(G)$, then $z=1$.
    \label{lem:diam3r=1}
\end{lemma}

\begin{lemma}
    If $G$ is a $(1, z, 3; -1)$-graph with $z \geq 2$ and $u \in V(G)$, then $r(u) = u^*$.
    \label{lem:diam3repstruct}
\end{lemma}
\begin{proof}
    Let $G$ be a $(1, z, 3; -1)$-graph with $z \geq 2$ and let $u \in V(G)$. By Theorem~\ref{thm:totallyregsummary}, $|N^-(u)| = |N^-(u^*)| = z+1$. Suppose for a contradiction that for each $v \in Z^-(u)$, a vertex in $N^-(u^*)$ appears in the $v^*$-branch of $T_v$. Then, for all $v \in Z^-(u)$, there are two $v^*, u^*$-paths of length at most three and $r(v^*) = u^*$. Therefore, at least two vertices in $G$ have repeat $u^*$, contradicting Lemma~\ref{lem:repperm}.

    So choose $v \in Z^-(u)$ such that no in-neighbour of $u^*$ appears in the $v^*$-branch of $T_v$. As $G$ has defect one, $v \nrightarrow u^*$. Since $|N^-(u^*)| = z+1$ and there are $z$ (directed) branches of $T_v$ other than the $v^*$-branch, there exists $w \in Z^+(v)$ with two in-neighbours of $u^*$ appearing in the $w$-branch of $T_v$. It follows that $r(w)=u^*$. Hence, by Lemma~\ref{lem:arrowsinclusion}, $u^*$ is a common in-neighbour of all vertices in $\arrows(w)$. If $w \neq u$, these arrow vertices are all repeated in $T_v$. But there are $z \geq 2$ such arrow vertices and so we conclude that $w=u$.
\end{proof}

\begin{theorem}
    \label{thm:almostmixedmoorenonexistencek=3}
    For $r,z \geq 1$, if $r > 1$ or $z > 1$, there is no $(r,z,3;-1)$-graph.
\end{theorem}
\begin{proof}
    First, by Corollary~\ref{cor:diam>=3r=1}, there is no $(r, z, 3; -1)$-graph whenever $r>1$. Suppose $z \geq 1$ and there exists a $(1,z,3;-1)$-graph, say $G$. By Lemma~\ref{lem:diam3repstruct}, $z=1$ or $z \geq 2$ and $r(u) = u^*$ for all $u \in V(G)$. But in the second case we contradict Lemma~\ref{lem:diam3r=1}. Hence, $z=1$.
\end{proof}

As earlier stated, there are three $(1,1,3;-1)$-graphs \cite{dalfo-2018}. The above result shows that these are the only almost mixed Moore graphs with diameter at least three.

\section{Conclusion}
We now conclude by stating some open problems. In Theorem~\ref{thm:defectboundz>1}, it was proven that the defect $\delta$ of an $(r,z,k;-\delta)$-graph $G$ is bounded below by $\sigmax(G) + 1$ unless $\delta = z = \sigmax(G)$. The bound would be more elegant and widely applicable if we could rule out the possibility of $\delta = z = \sigmax(G)$. The improved bound would imply the total regularity of $(r,1,k;-1)$-graphs, removing the need for Subsection~\ref{subsec:z=1totalregularity}. This motivates the following question.

\begin{problem}
    For which values of $r,z,k$ does there exist an $(r,z,k;-\delta)$-graph $G$ such that $\delta = z = \sigmax(G)$?
\end{problem}

We used our new bound to deduce that all $(r,z,k;-1)$-graphs with $z>1$ are totally regular. A counting argument was then used to prove the total regularity of $(r,1,k;-1)$-graphs with $r \geq 2$. Along with results from the existing literature, this completed the proof that all almost mixed Moore graphs are totally regular. We then used this result to prove that the $(1,1,3;-1)$-graphs of Figure~\ref{fig:almostmixedmooregraphs} are the only almost mixed Moore graphs with diameter at least three. This leaves the following open problem.

\begin{problem}[{See \cite[Quesitons 2, 3]{{lopez-2016}}}]
    Are there any almost mixed Moore graphs of diameter two other than the known $(2,1,2;-1)$-graph?
\end{problem}

Theorem~\ref{thm:r=z=1defectbound} gives a bound on the defect which increases with the diameter $k$. This result only applies to totally regular $(1,1,k;-\delta)$-graphs with $k \geq 3$. Intuitively, we would expect it to be possible to derive a similar bound under more general conditions. As suggested by James Tuite \cite{tuite-2026}, this motivates the following open problem.

\begin{problem}
    Is there a lower bound on the defect of a mixed graph which increases with the diameter?
\end{problem}

This work has focused on almost mixed Moore graphs. Since we have now settled the question of their existence for $k \geq 3$, an investigation into defect two mixed graphs would be a useful next step. We conclude by giving some questions which could be considered.

\begin{problem}
    Are all $(r,z,k;-2)$-graphs totally regular for $r, z \geq 1$, $k \geq 2$?
\end{problem}

\begin{problem}
    For which values of $r$, $z$, and $k$ does there exist an $(r,z,k;-2)$-graph?
\end{problem}

In this paper we have only allowed a single arc or edge between vertices, whereas in \cite{lopez-2016} there is no such constraint. Since we have focused on almost mixed Moore graphs of diameter at least three, this makes no difference to our earlier results. However, we cannot assume that this will be the case when studying mixed graphs of defect two.
Similarly to as discussed in \cite[Section 2]{lopez-2016}, since $M(r+2,0,2) = M(r,2,2) - 2$, examples of $(r,2,2;-2)$-graphs may be obtained by replacing certain edges by digons in known Moore graphs of diameter two. Since Moore graphs have been widely studied in their own right, we exclude examples arising in this way and ask the following problem.

\begin{problem}
    Does allowing digons affect the solutions to the previous two problems?
\end{problem}

\section*{Acknowledgements}
This work was completed as part of the author's dissertation for the Open University's MSc Mathematics. The author would like to thank his supervisor, Dr James Tuite, for his helpful feedback and advice. The author would also like to thank the M840 module team for their useful resources on producing a dissertation.

\end{document}